\documentclass[runningheads]{llncs}

\usepackage[T1]{fontenc}
\usepackage{amsmath,amssymb,mathtools}
\usepackage{enumitem}
\usepackage{hyperref}
\usepackage{xcolor}
\usepackage{tikz}

\usetikzlibrary{calc,decorations.pathreplacing,positioning}
\hypersetup{hidelinks}

\tikzset{
  v/.style={circle,draw=black,fill=white,line width=.55pt,
            minimum size=3.4mm,inner sep=0pt},
  dot/.style={circle,fill=black,draw=black,minimum size=2.2mm,inner sep=0pt},
  chosen/.style={circle,draw=black,double,fill=white,line width=.55pt,
                 minimum size=5.0mm,inner sep=0pt},
  edge/.style={draw=black,line width=.55pt},
  faint/.style={draw=black!45,line width=.45pt},
  set/.style={draw=black,dashed,line width=.55pt},
  clique/.style={draw=black,line width=.7pt},
  mathlabel/.style={font=\normalsize,fill=white,inner sep=1.2pt},
  note/.style={font=\footnotesize,align=center},
  smallnote/.style={font=\footnotesize,align=center},
  complete/.style={draw=black,line width=.8pt},
}

\newcommand{\NN}{\mathbb N}

\title{Extremal List Gaps and Inapproximability in Additive Graph Labeling}
\titlerunning{Extremal List Gaps in Additive Graph Labeling}
\author{Arash Ahadi \and Sharareh Alipour}
\authorrunning{A. Ahadi and S. Alipour}
\institute{Tehran Institute for Advanced Studies (TeIAS), Khatam University, Iran\\
 \email{aarash.ahadi.academic@gmail.com} \quad \email{sh.alipour@teias.institute}} 

\begin{document}
\maketitle
\begin{abstract}
We study a vertex-labeling analogue of the $1$-$2$-$3$ problem and its list version.  For a labeling $\ell:V(G)\to\mathbb N$, let $S_\ell(v)=\sum_{w\in N(v)}\ell(w)$.  The additive number $\eta(G)$ is the least $k$ for which there exists $\ell:V(G)\to[k]$ such that $S_\ell(u)\ne S_\ell(v)$ for every $uv\in E(G)$, while the list additive number $\eta_\ell(G)$ is the least $k$ such that the same condition can be satisfied from every assignment of $k$-element lists $L(v)\subset\mathbb N$ with $\ell(v)\in L(v)$.
We show that for every $k\ge2$, there is a graph $G$ with $\eta(G)=1$ and $\eta_\ell(G)\ge k$.  The separation persists at the minimum possible ordinary value for positive-degree regular graphs: there is a regular graph $H$ with $\eta(H)=2$ and $\eta_\ell(H)\ge k$.

We also determine a sharp lower bound for $\eta(G)$ in terms of the order and minimum degree of $G$, and show that the unbounded list gap persists at asymptotically extremal density.

Finally, for every fixed $k\ge2$, it is NP-hard to distinguish $\eta(G)=2$ from $\eta(G)>k$, even on asymptotically extremal dense graphs. Consequently, $\eta(G)$ admits no polynomial-time constant-factor approximation unless $\mathrm P=\mathrm{NP}$.

Together, these results reveal a robust gap phenomenon: the separation between ordinary and list additive labeling persists at the smallest possible ordinary values and even under asymptotically extremal density, while the ordinary parameter itself remains hard to approximate.
\end{abstract}

\section{Introduction}
Proper graph coloring distinguishes adjacent vertices by the labels assigned to them.  In \emph{additive graph labeling}, the labels matter only through their local sums: vertices receive positive integers, and the sums over their open neighborhoods must form a proper coloring.  More precisely, for a graph $G=(V,E)$ and a labeling $\ell:V(G)\to\NN$, write
$$
S_\ell(v)=\sum_{w\in N(v)}\ell(w).
$$
The labeling $\ell$ is an \emph{additive labeling} if $S_\ell(u)\ne S_\ell(v)$ for every edge $uv\in E(G)$.  The \emph{additive number} $\eta(G)$ is the least $k$ for which $G$ has an additive labeling with labels in $[k]=\{1, 2, ..., k\}$.  This model was introduced under the names \emph{lucky labeling} and \emph{lucky number} by Czerwi\'nski, Grytczuk, and \.{Z}elazny~\cite{CGZ09}. 
Upper bounds for $\eta$ have been studied on several graph classes~\cite{BartnickiBosekCzerwinskiGrytczukMateckiZelazny2014,Severin2020}.  
It is shown that computing additive number is NP-hard even on restricted planar instances~\cite{AhadiDehghanKazemiMollaahmadi2012}; moreover, deciding whether $\eta(G)=2$ is NP-complete even for $3$-regular graphs~\cite{DehghanSadeghiAhadi2018}. The inapproximability result of Ahadi and Dehghan~\cite{AhadiDehghan2016}
concerns a different objective, the minimum number of vertices labeled
$1$ in a $(0,1)$-additive labeling, rather than the ordinary additive
number $\eta(G)$ studied here.

Additive labeling is part of the broader study of distinguishing adjacent vertices by locally aggregated weights.  The best-known edge-weighting problem in this area is the $1$-$2$-$3$ Conjecture of Karo\'nski, \L{}uczak, and Thomason~\cite{KLT04}, recently proved by Keusch~\cite{Keu24}.  Its list form raises a basic robustness question: for every graph without an isolated edge, the fixed set $\{1,2,3\}$ suffices, but it remains open whether arbitrary prescribed three-element edge lists suffice, while five-element lists are known to do so~\cite{BartnickiGrytczukNiwczyk2009,Zhu2022}. The same question for vertex labels leads to the list version of additive labeling.

The \emph{list additive number} $\eta_\ell(G)$ is the least $k$ such that every assignment of $k$-element lists $L(v)\subset\NN$ admits an additive labeling with $\ell(v)\in L(v)$ for every vertex.  Akbari et al.~\cite{AkbariGhanbariManaviyatZare2013} began the systematic study of this parameter and proved the general bound $\eta_\ell(G)\le\Delta(G)^2-\Delta(G)+1$ for $\Delta(G)\ge2$. Further bounds for planar graphs are known, both under girth assumptions and in terms of maximum degree~\cite{BrandtTenpasYerger2020,LaiLih2022}. More recently, Gossett~\cite{Gossett2024} developed an Alon--Tarsi-type theorem for additive list coloring via graph orientations.
Ahadi and Dehghan~\cite{AhadiDehghan2016} proved that for every integer
$k\ge1$ there exists a graph $G$ such that
$$
  \eta(G)\le k\le \frac{\eta_\ell(G)}{2}.
$$
Thus the difference between the ordinary and list parameters can be
arbitrarily large, but their theorem does not keep the ordinary parameter
fixed as the gap grows. We show that the gap is unbounded even when the ordinary additive number
is fixed at one, and also when it is fixed at two for positive-degree
regular graphs.

In Section~\ref{2}, for every $k\ge2$, we construct a graph $G_k$ with
$$
  \eta(G_k)=1,
  \qquad
  \eta_\ell(G_k)\ge k.
$$
This proves that no function of the ordinary additive number can bound its list analogue. 

The gap persists even for regular graphs. 
For positive-degree regular graphs, the smallest possible additive number
is two.  For every $k\ge2$, we construct
$$
\text{a regular graph } H_k \text{ with}
\qquad
  \eta(H_k)=2,
  \qquad
  \eta_\ell(H_k)\ge k.
$$
All constructions are explicit.

We next give a sharp minimum-degree bound for graphs with a prescribed
additive number. For a graph $G$ of order $n$ and minimum degree $\delta$, we prove the sharp bound
$$
  \eta(G)\ge
  \frac{n}{(n-\delta)^2}
  +\frac12-\frac{1}{2(n-\delta)}.
$$
The bound is attained for every prescribed value of $\eta(G)$ and of $n-\delta$.  Moreover, any prescribed graph can be preserved as an induced subgraph of arbitrarily large asymptotically extremal graphs, with the same additive number and no smaller list additive number.

We also obtain an algorithmic consequence.  Section~\ref{sec:hardness} proves that, for every fixed $k\ge2$, it is NP-hard to distinguish $\eta(G)=2$ from $\eta(G)>k$. Consequently, unless $\mathrm P=\mathrm{NP}$, the additive number admits no polynomial-time constant-factor approximation; this strengthens the known hardness picture~\cite{AhadiDehghanKazemiMollaahmadi2012,DehghanSadeghiAhadi2018}.

\paragraph{Notation.}
All graphs are finite and simple.  We write $\overline G$ for the complement of $G$, and $\delta(G)$ and $\Delta(G)$ for its minimum degree and maximum degree, respectively.

\section{Unbounded list gaps at the minimum ordinary values} \label{2}

We begin by showing that the list additive number is unbounded even when the ordinary parameter is at its absolute minimum.  We then prove the analogous statement for positive-degree regular graphs, where the minimum ordinary value is two.

\begin{theorem}\label{thm:list-gap}
For every integer $k\ge2$, the following hold.
\begin{enumerate}[label=(\roman*)]
\item There exists a graph $G_k$ such that $\eta(G_k)=1$ and $\eta_\ell(G_k)\ge k$.
\item There exists a regular graph $H_k$ such that $\eta(H_k)=2$ and $\eta_\ell(H_k)\ge k$.
\end{enumerate}
\end{theorem}

\begin{proof}
For (i), put 
$$r=\max\{2,k-1\}, \qquad t=\lceil\log_2(r+1)\rceil.$$  
Create pairwise disjoint sets $S_0,\ldots,S_{t-1}$ of new vertices, where $|S_j|=2^j$; call these vertices \emph{selectors}, and put $X_j=\{2^j,2^j+1,\ldots,r2^j\}$, for every $0 \leq j \leq t-1$.

For $0\le i\le r$, let
$$
  i=\sum_{j=0}^{t-1}\varepsilon_j(i)2^j,
  \qquad \varepsilon_j(i)\in\{0,1\},
$$
be the binary representation of $i$.  For every vector $\mathbf x=(x_0,\ldots,x_{t-1})\in X_0\times\cdots\times X_{t-1}$, create a clique
\[
  C_{\mathbf x}=\{c^i_{\mathbf x}\,|\,0\le i\le r\}
\]
with $r+1$ new vertices.  For every $j$, join all vertices $c^i_{\mathbf x}$ with $\varepsilon_j(i)=1$ to all vertices of $S_j$, and add no other edges.  

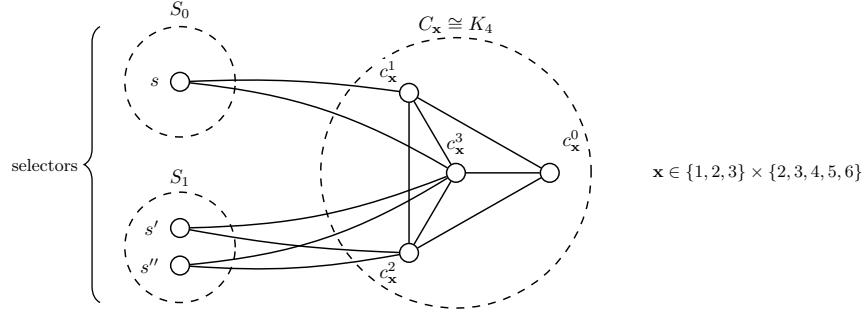
\begin{figure}[!t]
\centering
\vspace{0mm}
\begin{tikzpicture}[x=1cm,y=1cm,scale=0.73,transform shape]

\begin{scope}[yshift=0.45cm]

\draw[set] (-7.1,-2.95) circle [radius=1.0cm];
\node[mathlabel,anchor=south] at (-7.1,-1.85) {$S_0$};
\node[v] (s0) at (-7.1,-2.95) {};
\node[mathlabel,anchor=east] at (-7.43,-2.95) {$s$};

\draw[set] (-7.1,-5.95) circle [radius=1.0cm];
\node[mathlabel,anchor=south] at (-7.1,-4.88) {$S_1$};
\node[v] (s1a) at (-7.1,-5.60) {};
\node[v] (s1b) at (-7.1,-6.28) {};
\node[mathlabel,anchor=east] at (-7.43,-5.60) {$s'$};
\node[mathlabel,anchor=east] at (-7.43,-6.28) {$s''$};

\draw[decorate,decoration={brace,mirror,amplitude=6pt},line width=.5pt]
  (-8.55,-1.95) -- (-8.55,-6.95)
  node[midway,xshift=-1.0cm,note]{selectors};

\draw[set] (-2.10,-4.60) circle [radius=2.45cm];
\node[mathlabel,anchor=south] at (-2.10,-2.12) {$C_{\mathbf{x}}\cong K_4$};
\node[v] (c1) at (-2.95,-3.15) {};
\node[v] (c2) at (-2.95,-6.05) {};
\node[v] (c0) at (-0.40,-4.60) {};
\node[v] (c3) at (-2.10,-4.60) {};
\node[mathlabel,anchor=south east] at (-3.10,-3.00) {$c^1_{\mathbf{x}}$};
\node[mathlabel,anchor=north east] at (-3.10,-6.17) {$c^2_{\mathbf{x}}$};
\node[mathlabel,anchor=south] at (0.00,-4.25) {$c^0_{\mathbf{x}}$};
\node[mathlabel,anchor=south] at (-2.10,-4.30) {$c^3_{\mathbf{x}}$};
\foreach \a/\b in {c1/c2,c1/c0,c2/c0,c3/c1,c3/c2,c3/c0}
  \draw[edge] (\a)--(\b);

\draw[edge] (s0) to[bend left=5] (c1);
\draw[edge] (s0) to[bend left=12] (c3);
\draw[edge] (s1a) to[bend right=4] (c2);
\draw[edge] (s1a) to[bend right=10] (c3);
\draw[edge] (s1b) to[bend right=7] (c2);
\draw[edge] (s1b) to[bend right=13] (c3);

\node[note,anchor=west] at (1.35,-4.60)
  {$\mathbf{x}\in\{1,2,3\}\times\{2,3,4,5,6\}$};

\end{scope}

\end{tikzpicture}
\vspace{-3mm}
\caption{The construction of $G_4$, with $k=4$, $r=3$, and $t=2$. For every $\mathbf{x}\in\{1,2,3\}\times\{2,3,4,5,6\}$, there is a $4$-vertex clique $C_{\mathbf{x}}$ of the form shown in the figure. Hence $G_4$ has $15\cdot4+3=63$ vertices.}
\label{fig:section2-graph}
\vspace{-2mm}
\end{figure}

We first prove that $\eta(G_k)=1$.  Each $c^i_{\mathbf x}$ has $r$ neighbors in its clique and exactly $i$ selector neighbors, so $d(c^i_{\mathbf x})=r+i$.  Thus the clique vertices have degrees $r,r+1,\ldots,2r$.

Every selector has at least one neighbor in every clique.  The number of cliques is
\[
  M=\prod_{j=0}^{t-1}\bigl((r-1)2^j+1\bigr),
\]
so every selector has degree at least $M$.  Since $t\ge2$ and $r\ge2$, we have $M\ge r(2r-1)>2r$.  Hence every selector has larger degree than every clique vertex. Consequently, the endpoints of every edge have different degrees. Therefore the all-one labeling is an additive labeling, so $\eta(G_k)=1$.

It remains to show that $\eta_\ell(G_k)>r$.
Define the lists 
\[
\begin{aligned}
L(s) &= [r],
&& \text{for every selector } s,\\
L(c^i_{\mathbf x}) &= B_i(\mathbf x)+[r],
&& \text{for every clique vertex } c^i_{\mathbf x}.
\end{aligned}
\]
where
\[
B_i(\mathbf x)=\sum_{j=0}^{t-1}\varepsilon_j(i)x_j.
\]
To the contrary, suppose that an additive labeling $\ell$ were chosen from these lists.  For each $j$, let $\sigma_j=\sum_{s\in S_j}\ell(s)$. Since $|S_j|=2^j$ and all selector labels lie in $[r]$, we have $\sigma_j\in X_j$. Hence $\mathbf x^*=(\sigma_0,\ldots,\sigma_{t-1})$ indexes one of the cliques in the construction.

Consider $C_{\mathbf x^*}$. By the definition of the lists, for every $0\le i\le r$, there exists $z_i\in[r]$ such that \begin{equation}\label{eq:clique-label}
  \ell(c^i_{\mathbf x^*})=B_i(\mathbf x^*)+z_i.
\end{equation}

Put $Q=\sum_{h=0}^r\ell(c^h_{\mathbf x^*})$. 
By construction, $c^i_{\mathbf x^*}$ is adjacent to every vertex of $S_j$
exactly when $\varepsilon_j(i)=1$. Therefore the sum of the labels of its
selector neighbors is
\[
  \sum_{j=0}^{t-1}\Big(\varepsilon_j(i)
  \sum_{s\in S_j}\ell(s)\Big)
  =
  \sum_{j=0}^{t-1}\varepsilon_j(i)\sigma_j
  =
  B_i(\mathbf x^*).
\]
 Hence
 \[
S_\ell(c^i_{\mathbf x^*})
=
\sum_{h\ne i}\ell(c^h_{\mathbf x^*})+
\sum_{\substack{s\text{ is a selector and }\\  s\in N(c^i_{\mathbf x^*})}}\ell(s)
=\bigl(Q-\ell(c^i_{\mathbf x^*})\bigr)+
B_i(\mathbf x^*)
=
Q-z_i,
\]
the last equality holds by (\ref{eq:clique-label}).

Thus the $r+1$ vertices of $C_{\mathbf x^*}$ can have pairwise distinct neighbor-sums only if $z_0,\ldots,z_r$ are pairwise distinct, which is impossible because they all belong to the $r$-element set $[r]$.  Hence this $r$-list assignment admits no additive labeling.  Therefore $\eta_\ell(G_k)>r\ge k-1$, and hence $\eta_\ell(G_k)\ge k$.

For (ii), retain $r,t$ and the binary digits $\varepsilon_j(i)$ from part (i), and take selector sets $S_0,\ldots,S_{t-1}$ with $|S_j|=2^j$ as before.  For
$0\le j\le t-1$, put \[
Y_j=\{-(r-1)2^j,\ldots,(r-1)2^j\},
\]
and let $\mathcal Y=Y_0\times\cdots\times Y_{t-1}$.
As in part (i), for $\mathbf y=(y_0,\ldots,y_{t-1})\in\mathcal Y$, write
$$
  B_i(\mathbf y)=\sum_{j=0}^{t-1}\varepsilon_j(i)y_j;
$$
and create a clique $C_{\mathbf y}=\{c^0_{\mathbf y},\ldots,c^r_{\mathbf y}\}$
with $r+1$ new vertices, and exactly as in part (i), join $c^i_{\mathbf y}$ to
every vertex of $S_j$ whenever $\varepsilon_j(i)=1$.  Call the
resulting graph $J$.

As in part (i), every clique vertex $c^i_{\mathbf y}$ has degree
$r+i$, while every selector has degree greater than $2r$.  Thus the
endpoints of every edge of $J$ have different degrees.  Add sufficiently
many isolated vertices to $J$, if necessary, so that, writing
$N=|V(J)|$,
$$
  N-1>2\Delta(J).
$$
This does not change any degree of a nonisolated vertex.

Take two disjoint copies $J^\alpha$ and $J^\beta$ of $J$.  In addition to the
edges inside the two copies, join $u^\alpha$ to $v^\beta$ precisely when
$$
  u\ne v
  \qquad\text{and}\qquad
  uv\notin E(J).
$$
Call the resulting graph $H_k$.  For $\gamma\in\{\alpha,\beta\}$,
$$
  d_{H_k}(v^\gamma)
  =d_J(v)+\big(N-1-d_J(v)\big)
  =N-1.
$$
Hence $H_k$ is $(N-1)$-regular.

We next prove that $\eta(H_k)=2$.  Label every vertex of $J^\alpha$ by $1$
and every vertex of $J^\beta$ by $2$.  For $v\in V(J)$,
$$
  S(v^\alpha)=2(N-1)-d_J(v),
  \qquad
  S(v^\beta)=N-1+d_J(v).
$$
If $uv\in E(J)$, then $d_J(u)\ne d_J(v)$, so every edge lying inside
one of the two copies is distinguished.  Moreover, if $u^\alpha v^\beta$ is a cross edge,
then
$$
  S(u^\alpha)-S(v^\beta)
  =
  N-1-d_J(u)-d_J(v)
  >0,
$$
since $N-1>2\Delta(J)$.  Thus, this is an additive labeling with labels
$1$ and $2$.  Since $H_k$ is regular and has edges, the all-one labeling is
not additive, and therefore
$$
  \eta(H_k)=2.
$$
It remains to prove that $\eta_\ell(H_k)>r$.  Give every vertex of
$H_k$ the list $[r]$, except that for
$\mathbf y\in\mathcal Y$ and $0\le i\le r$, assign
$$
  L(c^{i,\alpha}_{\mathbf y})
  =
  r(r-1)+B_i(\mathbf y)+[r],
$$
where $c^{i,\alpha}_{\mathbf y}$ is the vertex $c^i_{\mathbf y}$ in $J^\alpha$.  These are lists of positive integers, since
$$
  B_i(\mathbf y)\ge -(r-1)i\ge-r(r-1).
$$
Suppose, for a contradiction, that an additive labeling $\ell$ is
chosen from these lists.  For each $j$, put
$$
  y_j^*
  =
  \sum_{s\in S_j}\ell(s^\alpha)
  -
  \sum_{s\in S_j}\ell(s^\beta),
$$
where, for $\gamma\in\{\alpha,\beta\}$, $s^\gamma$ denotes the copy of $s$ in $J^\gamma$. Since $|S_j|=2^j$ and all selector labels belong to $[r]$,
$$
  -(r-1)2^j\le y_j^*\le(r-1)2^j,
$$
and hence
$$
  \mathbf y^*=(y_0^*,\ldots,y_{t-1}^*)\in\mathcal Y.
$$
Therefore the clique $C_{\mathbf y^*}$ occurs in the construction, and hence
$C_{\mathbf y^*}^\alpha$ occurs in $J^\alpha$.

For every $0\le i\le r$, write
$$
  \ell(c^{i,\alpha}_{\mathbf y^*})
  =
  r(r-1)+B_i(\mathbf y^*)+z_i,
  \qquad z_i\in[r].
$$
Put
$$
  Q^\gamma=\sum_{h=0}^r\ell(c^{h,\gamma}_{\mathbf y^*}), \qquad P=\sum_{v\in V(J)}\ell(v^\beta),
$$
for $\gamma\in\{\alpha,\beta\}$.

The internal clique contribution to
$S_\ell(c^{i,\alpha}_{\mathbf y^*})$ is
$Q^\alpha-\ell(c^{i,\alpha}_{\mathbf y^*})$.  By the definition of the cross edges, the contribution from
$J^\beta$ is
$$
  P-Q^\beta-
  \sum_{j=0}^{t-1}\varepsilon_j(i)
  \sum_{s\in S_j}\ell(s^\beta).
$$
Together with the selector contribution from $J^\alpha$, this gives
$P-Q^\beta+B_i(\mathbf y^*)$.  Hence
$$
\begin{aligned}
  S_\ell(c^{i,\alpha}_{\mathbf y^*})
  &=
  Q^\alpha-\ell(c^{i,\alpha}_{\mathbf y^*})
  +B_i(\mathbf y^*)+P-Q^\beta\\
  &=
  Q^\alpha+P-Q^\beta-r(r-1)-z_i.
\end{aligned}
$$
Thus the $r+1$ vertices of $C_{\mathbf y^*}^\alpha$ can have pairwise
distinct neighbor-sums only if $z_0,\ldots,z_r$ are pairwise distinct,
which is impossible because they all belong to the $r$-element set
$[r]$.  Hence $\eta_\ell(H_k)>r\ge k-1$,
and therefore $\eta_\ell(H_k)\ge k$.
\qed \end{proof}

As an immediate consequence of Theorem~\ref{thm:list-gap}(i) and Theorem~\ref{thm:dense-lift}, we obtain the following.
\begin{corollary}\label{cor:all-smaller-lucky}
For every pair of integers $a\ge1$ and $b\ge2$, there are arbitrarily
large graphs $G$ of order $n$ such that
\[
  \eta(G)=a,\qquad \eta_\ell(G)\ge b,
\qquad
\text{and }
\quad
  \delta(G)
  =n-\sqrt{\frac{2n}{2a-1}}+O(1).
\]
\end{corollary}

\begin{proof}
Let $G_b$ be the graph from Theorem~\ref{thm:list-gap}(i).  If $a=1$,
put $H=G_b$.  If $a\ge2$, put $H=G_b\cup K_a$.  The additive number of a
disjoint union is the maximum of its values on the components, so
$\eta(H)=a$, while $\eta_\ell(H)\ge\eta_\ell(G_b)\ge b$.  Apply
Theorem~\ref{thm:dense-lift} to $H$ with $\eta(H)=a$.
\qed \end{proof}

\section{Exact Extremal density}
\label{sec:dense}
To study how dense such examples can be, we first determine the exact
complement-degree extremal bound for graphs with prescribed additive number.

\begin{theorem}[Sharp complement-degree bound]\label{thm:complement-degree-bound}
Let $G$ be a graph of order $n$ and minimum degree $\delta$. Then
\[
  \eta(G)\ge
  \frac{n}{(n-\delta)^2}
  +\frac12-\frac{1}{2(n-\delta)}.
\]
Moreover, the bound is sharp: for every pair of integers $a\ge1$ and
$b\ge1$, there exists a graph $G$ with $\eta(G)=a$ and
$n-\delta=b$ attaining equality.
\end{theorem}

\begin{proof}
Put $a=\eta(G)$ and $D=n-1-\delta=\Delta(\overline G)$, and fix an
additive labeling $\ell:V(G)\to[a]$.
Let $L=\sum\limits_{v}\ell(v)$,
and for every $v\in V(G)$ define
$$
  q(v)=\ell(v)+\sum_{u\in N_{\overline G}(v)}\ell(u).
$$
Since every vertex other than $v$ is adjacent to $v$ in exactly one
of $G$ and $\overline G$, we have
$$
  S_\ell(v)=L-q(v).
$$
Consequently, if $uv\in E(G)$, then the additive-labeling property gives $q(u)\ne q(v)$.

For every integer $j$, let
$$
  V_j=\{v\in V(G):q(v)=j\}.
$$
The set $V_j$ is independent in $G$, and hence induces a clique in
$\overline G$. Therefore
$$
  |V_j|\le D+1.
$$
Moreover, if $v\in V_j$, then every other vertex of $V_j$ is a neighbor
of $v$ in $\overline G$, and hence
$$
  j=q(v)\ge\sum_{u\in V_j}\ell(u)\ge |V_j|.
$$
Thus
$$
  |V_j|\le\min\{j,D+1\}.
$$
On the other hand, since $\ell(v)\in[a]$ and
$d_{\overline G}(v)\le D$ for every vertex $v$,
$$
  1\le q(v)\le a(D+1).
$$
Hence
$$
  n
  =\sum_j |V_j|
  \le \sum_{j=1}^{D+1}j
  +\sum_{j=D+2}^{a(D+1)}
  (D+1).
$$
Therefore
$$
  n
  \le a(D+1)^2-\frac{D(D+1)}2.
$$
Since $D+1=n-\delta$, rearranging gives
\[
  \eta(G)=a\ge
  \frac{n}{(n-\delta)^2}
  +\frac12-\frac{1}{2(n-\delta)}.
\]

For sharpness, fix $a\ge1$ and $b\ge1$, and put $D=b-1$. Construct a complete multipartite graph with parts
$$
  A_1,\ldots,A_{a(D+1)},
  \qquad
  |A_j|=\min\{j,D+1\}.
$$
Assign labels from $[a]$ to the vertices of $A_j$ so that the total
label in that part is $j$.  This is possible because
$$
  |A_j|\le j\le a|A_j|.
$$
Explicitly, if $j\le D+1$, use label $1$ throughout.  Otherwise, start
from all labels $1$ and distribute $j-(D+1)$ extra units among the
$D+1$ vertices, with at most $a-1$ extra units at each vertex.

If $L$ denotes the total of all labels, each vertex in $A_j$ has
neighborhood sum $L-j$.  Different parts have different sums, while
vertices in the same part are nonadjacent.  Thus $\eta(G)\le a$.

The largest part has size $D+1$, so $n-\delta=D+1=b$.  The total
number of vertices is
$$
  |V(G)|
  =\sum_{j=1}^{D+1}j+(a-1)(D+1)^2
  =a(D+1)^2-\frac{D(D+1)}2.
$$
For $a\ge2$, an additive labeling from $[a-1]$ would contradict the
bound proved above with parameter $a-1$, since it would imply
$$
  |V(G)|
  \le (a-1)(D+1)^2-\frac{D(D+1)}2.
$$
Therefore $\eta(G)=a$.  For $a=1$, the displayed construction already
gives $\eta(G)=1$.
\qed \end{proof}

The extremal construction above does not necessarily preserve a given
induced subgraph.  The following theorem gives an asymptotically extremal
extension that does, while preserving the additive number and the relevant
lower bound on the list additive number.  We will also use this extension
in the hardness reduction.

\begin{theorem}[Dense lift]\label{thm:dense-lift}
Let $H$ be a graph of order $h$ with $\eta(H)=a$. Then $H$ has arbitrarily large
induced extensions $G$ such that
\[
  \eta(G)=a,\qquad
  \eta_\ell(G)\ge\eta_\ell(H), \qquad \text{and } \quad \delta(G)
  =n-\sqrt{\frac{2n}{2a-1}}+O_H(1),
\]
where $n=|V(G)|$. Moreover, the displayed minimum-degree estimate is asymptotically best
possible among graphs with additive number $a$.
\end{theorem}

\begin{proof}
Fix $t>ah$.  Take an induced copy of $H$.  For every integer
$j$ with $ah+1\le j\le at$, create an independent set $A_j$ of order
\[
  |A_j|=\min\{j,t\}.
\]
Join every vertex of $H$ to every vertex of each $A_j$, and join every
vertex of $A_i$ to every vertex of $A_j$ whenever $i\ne j$.  Add no
other edges, and call the resulting graph $G$.
The order of $G$ is
\[
  n
  =h+\sum_{j=ah+1}^{t}j+(a-1)t^2
  =h+at^2-\frac{t(t-1)}2-\frac{ah(ah+1)}2.
\]
Every vertex outside $H$ is adjacent to every vertex of $H$.  Thus,
for any labeling $\ell$ of $G$ and any $u,v\in V(H)$,
\[
  S_\ell^G(u)-S_\ell^G(v)
  =
  S_{\ell|_H}^H(u)-S_{\ell|_H}^H(v),
\]
where the superscript indicates the graph in which the neighborhood
sum is taken.  Consequently,
\[
  \eta(G)\ge\eta(H)=a.
\]
If $\eta_\ell(H)>1$, extend a bad
$(\eta_\ell(H)-1)$-list assignment on $H$ by arbitrary lists of the
same size on the new vertices.  An additive labeling from these lists
would restrict, by the displayed identity, to an additive labeling of
$H$, a contradiction.  Hence
\[
  \eta_\ell(G)\ge\eta_\ell(H),
\]
with the case $\eta_\ell(H)=1$ being immediate.

For the reverse inequality, fix an additive labeling
$\ell:V(H)\to[a]$.  For $v\in V(H)$, put
\[
  q(v)=\sum_{x\in V(H)}\ell(x)-S_\ell^H(v).
\]
Then $q(v)\le a|V(H)|=ah$, and $q(u)\ne q(v)$ whenever $uv\in E(H)$.
For every remaining part $A_j$, choose labels from $[a]$ whose sum is
$j$, as in the equality construction of
Theorem~\ref{thm:complement-degree-bound}; this is possible since
\[
  |A_j|\le j\le a|A_j|.
\]
Let $L$ be the sum of all labels in $G$.  Then
\[
  S_\ell^G(v)=L-q(v)
  \quad (v\in V(H)),
  \qquad
  S_\ell^G(x)=L-j
  \quad (x\in A_j).
\]
Edges inside $H$ are distinguished because the values $q(v)$ are
distinct on adjacent vertices.  An edge between $H$ and $A_j$ is
distinguished because
\[
  q(v)\le ah<j,
\]
and vertices in distinct parts $A_i$ and $A_j$ have different neighborhood
sums because $i\ne j$.  Hence $\eta(G)\le a$, and therefore
$\eta(G)=a$.

Finally, in $\overline G$ the remaining parts induce disjoint cliques,
while $\overline H$ is another component.  Every such component has
order at most $t$, and $A_t$ has order exactly $t$.  Thus
\[
  \Delta(\overline G)=t-1,
\]
so $\delta(G)=n-t$.

From the exact order formula,
\[
  n=\frac{2a-1}{2}t^2+O_H(t),
\]
and therefore
\[
  t=\sqrt{\frac{2n}{2a-1}}+O_H(1).
\]
Hence
\[
  \delta(G)
  =n-\sqrt{\frac{2n}{2a-1}}+O_H(1).
\]
Since $t$ can be arbitrarily large, so can $n$.  The asymptotic
optimality follows from Theorem~\ref{thm:complement-degree-bound}.
\qed
\end{proof}

\section{Hardness of approximation on almost complete graphs}\label{sec:hardness}
In this section, we study the computational complexity of determining $\eta(G)$ and, as in Section~\ref{2}, establish an arbitrarily large gap. Our reduction uses the following consequence of the proof of
Zuckerman's inapproximability theorem for the chromatic
number~\cite{Zuc07}: for every fixed $\varepsilon>0$, it is NP-hard,
on sufficiently large $N$-vertex graphs $F_0$, to distinguish between
$\chi(F_0)\le N^\varepsilon$ and
$\chi(F_0)>N^{1-\varepsilon}$.

\begin{theorem}\label{thm:fixed-lucky-gap}
For every fixed integer $k\ge2$, it is NP-hard to distinguish graphs
$G$ satisfying $\eta(G)=2$ from graphs satisfying
$\eta(G)>k$, even when
\[
  \delta(G)
  =n-\sqrt{\frac{2n}{3}}+O(1).
\]
For the yes-instances, the displayed minimum-degree estimate is
asymptotically optimal.
\end{theorem}

\begin{proof}
Fix $k\ge2$, and choose a constant $\varepsilon>0$ so small that $\varepsilon k<1-\varepsilon$.
Let $F_0$ be an $N$-vertex instance of the coloring gap, and let $F$ be obtained by adding one universal vertex and joining it to all vertices of $F_0$.  Then $F$ is connected and \[\chi(F)=\chi(F_0)+1.\]
Put $a=\lceil N^\varepsilon\rceil+1$. We construct a graph $H=H(F,a,k)$.

This graph has $k$ layers. For each $v\in V(F)$ and each layer $p\in[k]$, create a set $B_{v,p}$ of $a-1$ vertices, called \emph{blocks}.  Fix an arbitrary orientation of every edge of $F$, and write $uv$ when the edge is oriented from $u$ to $v$.  For every edge $uv\in E(F)$, create a clique
\[
  C_{uv}=\{c^j_{uv}\,|\,0\le j\le k\}
\]
of order $k+1$.  For $0\le j\le k$, join $c^j_{uv}$ to every vertex in
\[
  B_{u,1}\cup\cdots\cup B_{u,j}
 \,\,\,\, \qquad\text{and}\qquad\,\,\,\,
  B_{v,j+1}\cup\cdots\cup B_{v,k}.
\]
Finally, attach $R=ka$
private leaves to every block vertex.  Since $k$ is fixed and $a=O(N^\varepsilon)$, the construction has polynomial size.

Suppose first that $\chi(F)\le a$, and fix a proper coloring
\[
  \rho:V(F)\to\{0, 1, ..., a-1\}.
\]
In every block $B_{v,p}$, label exactly $\rho(v)$ vertices by $2$ and the remaining vertices by $1$.  Thus the sum of the labels in each such block is $(a-1)+\rho(v)$.
Label every clique vertex by $1$ and every private leaf by $2$.  For $uv\in E(F)$, the contribution of the block vertices to the neighborhood sum of $c^j_{uv}$ is
\[
\begin{aligned}
&j\bigl((a-1)+\rho(u)\bigr)
 +(k-j)\bigl((a-1)+\rho(v)\bigr)\\
={}& k(a-1+\rho(v))
 +j\bigl(\rho(u)-\rho(v)\bigr).
\end{aligned}
\]
These $k+1$ quantities are pairwise distinct because $\rho(u)\ne\rho(v)$.  Since every clique vertex has $k$ neighbors in $C_{uv}$, all labeled $1$, we have
\[
S(c^j_{uv})
=k(a+\rho(v))+j\bigl(\rho(u)-\rho(v)\bigr)
\le 2ka-k
<2ka.
\]
On the other hand, every block vertex has $R=ka$ private leaves, all labeled $2$, and hence has neighborhood sum at least $2R=2ka$.
Thus every edge between a block vertex and a clique vertex is distinguished.  Finally, the neighborhood sum of every private leaf is the label of its parent and hence belongs to $\{1,2\}$, whereas the neighborhood sum of its parent is at least $2ka>2$.  Thus every leaf edge is also distinguished. Hence $\eta(H)\le2$.
Every vertex of a clique $C_{uv}$ has degree
$k+k(a-1)=ka$.  Thus the all-one labeling fails on every clique edge; consequently
\[
\eta(H)=2.
\]
For soundness, suppose that $\eta(H)\le k$, and fix an additive labeling
$\ell:V(H)\to[k]$.  For every $v\in V(F)$ and $p\in[k]$, define
\[
  \sigma_p(v)=\sum_{x\in B_{v,p}}\ell(x),
  \qquad
  \sigma(v)=\bigl(\sigma_1(v),\ldots,\sigma_k(v)\bigr).
\]
We claim that $\sigma$ is a proper coloring of $F$.  Suppose instead
that $uv\in E(F)$ and $\sigma(u)=\sigma(v)$.  For every
$0\le j\le k$, the contribution of the block neighbors to the
neighborhood sum of $c^j_{uv}$ is
\[
  \sum_{p\le j}\sigma_p(u)
  +\sum_{p>j}\sigma_p(v)
  =\sum_{p=1}^k\sigma_p(u),
\]
which is independent of $j$.  Among the $k+1$ vertices of $C_{uv}$,
two receive the same label from $[k]$.  Their internal clique
contributions are therefore equal, and the displayed identity shows
that their block contributions are equal as well.  Hence their
neighborhood sums are equal, a contradiction.

Thus $\sigma$ properly colors $F$.  Each $\sigma_p(v)$ is the sum of
$a-1$ labels from $[k]$, and therefore has at most
\[
  (k-1)(a-1)+1
\]
possible values.  Consequently,
\[
  \chi(F)
  \le \bigl((k-1)(a-1)+1\bigr)^k
  \le (ka)^k.
\]
For all sufficiently large $N$, since
$a=\lceil N^\varepsilon\rceil+1$ and $\varepsilon k<1-\varepsilon$,
we have
\[
  (ka)^k<N^{1-\varepsilon}.
\]
Hence in the no-case of the chromatic gap the constructed graph cannot
satisfy $\eta(H)\le k$.  Therefore $\eta(H)>k$.

It remains to impose the dense structural restrictions.  Put
$h=|V(H)|$, and apply the extension construction used in the proof of
Theorem~\ref{thm:dense-lift}, with the formal parameter $a=2$ and
$t=2h^2$.  Only the construction, rather than the conclusion of
Theorem~\ref{thm:dense-lift}, is used here: it is defined for every seed
graph $H$.  Independently of the value of $\eta(H)$, the
resulting graph $G$ contains $H$ as an induced subgraph and satisfies
\[
  S_\ell^G(u)-S_\ell^G(v)
  =
  S_{\ell|_H}^H(u)-S_{\ell|_H}^H(v)
  \qquad (u,v\in V(H))
\]
for every labeling $\ell$ of $G$.  Consequently,
\[
  \eta(G)\ge\eta(H).
\]
The order and minimum degree calculations in that construction are also
independent of $\eta(H)$.  Since $t=2h^2$, they give
\[
  n
  =
  h+\sum_{j=2h+1}^{t}j+t^2
  =
  \frac32t^2-\frac12t
\]
and $\delta(G)=n-t$.
Therefore
\[
  t=\sqrt{\frac{2n}{3}}+O(1),
  \qquad
  \delta(G)
  =
  n-\sqrt{\frac{2n}{3}}+O(1).
\]
In the yes-case, $\eta(H)=2$, and the labeling argument in the proof of
Theorem~\ref{thm:dense-lift} gives $\eta(G)\le2$; hence
$\eta(G)=2$.  In the no-case, the restriction property above gives
$\eta(G)\ge\eta(H)>k$.  Finally, since $t=O(h^2)$ and the extension has
order $O(t^2)$, the construction remains polynomial.  The asymptotic
optimality in the yes-case follows from
Theorem~\ref{thm:complement-degree-bound}.  The theorem follows.
\qed \end{proof}

\begin{corollary}\label{cor:no-constant-approximation}
Unless $\mathrm{P}=\mathrm{NP}$, $\eta(G)$ admits no polynomial-time
constant-factor approximation, even when restricted to graphs with $\delta(G)=n-O(\sqrt n)$.
\end{corollary}
\begin{proof}
A $C$-approximation returns an additive labeling whose largest label is
at most $C\eta(G)$. Suppose that a polynomial-time $C$-approximation exists for some constant
$C\ge1$. Choose an integer $k>2C$.  On a yes-instance, the
largest label returned by the algorithm is at most $2C<k$, whereas on a
no-instance every feasible additive labeling has largest label greater
than $k$.  Thus the algorithm would distinguish the two cases in
Theorem~\ref{thm:fixed-lucky-gap}.
\qed \end{proof}

\section*{Declaration of Generative AI and AI-assisted Technologies}

During the development of this work, the authors used OpenAI's ChatGPT as an assistive tool for brainstorming ideas, exploring proof strategies, and checking arguments for possible gaps or errors.  All mathematical statements, proofs, and conclusions were independently reviewed and verified by the authors.  The authors take full responsibility for the accuracy, originality, and integrity of the work.

\end{document}